\documentclass[11pt,a4paper,leqno]{article}
\usepackage[margin=1.1in]{geometry}
\usepackage{amsmath,amsthm,amssymb,enumerate, bbm ,graphicx,color,caption,upgreek, float, tikz, wasysym, subcaption,booktabs,longtable, appendix,graphics, pdfpages,rotating,mathtools,textcomp, array, blkarray, 
	courier,
	 tkz-graph,
comment}

\usepackage{tikz}
\usetikzlibrary{arrows,shapes}

\usepackage[justification=justified,singlelinecheck=false]{caption}

\definecolor{donkergroen}{RGB}{46,148,0}
\usepackage[hidelinks]{hyperref} 
\definecolor{donkerrood}{RGB}{204,0,0}

\definecolor{blauw}{RGB}{61,158,255}
\definecolor{donkerblauw}{RGB}{0,0,255}
\definecolor{donkergroen}{RGB}{46,148,0}
\definecolor{donkerrood}{RGB}{204,0,0}

\usepackage{abstract}

\makeatletter
\newif\if@borderstar
\def\bordermatrix{\@ifnextchar*{%
\@borderstartrue\@bordermatrix@i}{\@borderstarfalse\@bordermatrix@i*}%
}
\def\@bordermatrix@i*{\@ifnextchar[{\@bordermatrix@ii}{\@bordermatrix@ii[()]}}
\def\@bordermatrix@ii[#1]#2{%
\begingroup
\m@th\@tempdima8.75\p@\setbox\z@\vbox{%
\def\cr{\crcr\noalign{\kern 2\p@\global\let\cr\endline }}%
\ialign {$##$\hfil\kern 2\p@\kern\@tempdima & \thinspace %
\hfil $##$\hfil && \quad\hfil $##$\hfil\crcr\omit\strut %
\hfil\crcr\noalign{\kern -\baselineskip}#2\crcr\omit %
\strut\cr}}%
\setbox\tw@\vbox{\unvcopy\z@\global\setbox\@ne\lastbox}%
\setbox\tw@\hbox{\unhbox\@ne\unskip\global\setbox\@ne\lastbox}%
\setbox\tw@\hbox{%
$\kern\wd\@ne\kern -\@tempdima\left\@firstoftwo#1%
\if@borderstar\kern2pt\else\kern -\wd\@ne\fi%
\global\setbox\@ne\vbox{\box\@ne\if@borderstar\else\kern 2\p@\fi}%
\vcenter{\if@borderstar\else\kern -\ht\@ne\fi%
\unvbox\z@\kern-\if@borderstar2\fi\baselineskip}%
\if@borderstar\kern-2\@tempdima\kern2\p@\else\,\fi\right\@secondoftwo#1 $%
}\null \;\vbox{\kern\ht\@ne\box\tw@}%
\endgroup
}
\makeatother

\makeatletter 
\newcommand\mynobreakpar{\par\nobreak\@afterheading} 
\makeatother

\newcommand{\N}{\mathbb{N}}
\newcommand{\Z}{\mathbb{Z}}

\usepackage[english]{babel}
\newtheorem{conjecture}{Conjecture}[section]
\newtheorem{theorem}{Theorem}[section]
\newtheorem{lemma}[theorem]{Lemma}
\newtheorem{claim}[theorem]{Claim}
\newtheorem{proposition}[theorem]{Proposition}

\theoremstyle{definition}
\newtheorem{defn}[theorem]{Definition}

\newtheorem{observation}[theorem]{Observation}

\newtheorem*{examp*}{Example}

\newtheorem{remark}[theorem]{Remark}

\theoremstyle{plain}

\makeatletter
\let\@fnsymbol\@arabic
\makeatother

\newfloat{Algorithm}{!hbt}{alg}

\newcounter{thm}[section]

\title{A topological characterization of Gauss codes}

\author{
	Llu\'{i}s Vena\thanks{Korteweg-De Vries Institute for Mathematics, University of Amsterdam. \texttt{lluis.vena@gmail.com}. The research leading to these
results has received funding from the European Research Council under the European Union’s Seventh Framework Programme (FP7/2007-2013) / ERC grant agreement \textnumero 339109.} 
}
\date{}

\begin{document}

	\maketitle
	
	\begin{abstract}
	A (smooth) embedding of a closed curve on the plane with finitely many intersections is said to be \emph{generic} if each point of self-intersection is crossed exactly twice and at non-tangent angles.
	A finite word $\omega$ where each character occurs twice  is a \emph{Gauss code} if it can be obtained as the sequence of traversed self-intersections of a generic plane embedding of a close curved $\gamma$. Then $\gamma$ is said to realize $\omega$.
	

We present a characterization of Gauss codes using Seifert cycles. The characterization is
 given by an algorithm that, given a Gauss code $\omega$ as input, it outputs a combinatorial plane embedding of closed curve in linear time with respect to the number of characters of the word. The algorithm allows to find all the (combinatorial) embeddings of a closed curve in the plane that realize $\omega$.
 
The characterization involve two functions between graphs embedded on orientable surfaces, invertible of each other. One produces an embedding of the Seifert graph of the word (or paragraph), and the other is its inverse operation. These operations might be of independent interest.

	\end{abstract}

\section{Introduction}

A closed curve embedded on the plane with finitely many self-intersections (all of which are traversed exactly twice at non-tangent angles), give rise to a word where every character appears exactly twice using the following procedure.  Label each intersecting point with a different character of an alphabet, and record the characters traversed when moving along the curve with a certain direction from a certain starting point.  See Figure~\ref{fig.5} as an example.

Gauss \cite[pg. 282--286]{gauss00} asked to characterize which finite double occurrence words on an alphabet arise in this way. If a word can be realized as the self-intersections of a closed curve on the plane, then it is said to be a \emph{Gauss code}, or a \emph{realizable word}.

Gauss gave a necessary condition for a word to be a realizable: between any two copies of the same character, there should be an even number of characters, counted with multiplicity. However, the condition is not sufficient. For instance, the double occurrence word $1234534125$ satisfy Gauss' condition, yet no plane embedding of a closed curve induces it.
In \cite{nagy27}, Nagy gave additional necessary conditions for a word to be realizable. Necessary and sufficient conditions for the characterizations of Gauss codes, can be deduced from the work of Dehn
\cite{dehn1936kombinatorische} (see \cite{rosenstiehl1976gauss}). 

Several other characterizations, simplifications, reformulations and proofs have been found. For instance, in \cite{lovasz1976forbidden}, Lov\'asz and Marx present a characterization in terms of forbidden structures. In \cite{rosenstiehl1978principal}, Ronsestiehl and Read give a characterization in terms of properties of the interlacing graph (the vertices being the characters and two vertices are connected if the word contains a copy of one character between the two copies of the other), and the proof is algebraic. In \cite{de1999characterization}, De Fraysseix and Ossona de Mendez give a reformulation of the characterization by Rosenstiehl et al. \cite{rosenstiehl1978principal,rosenstiehl1978graphes} using an operation on words. See \cite{treybig1968characterization,godsil2013algebraic,rosenstiehl1976gauss,rosenstiehl1999new} for other characterizations and relevant works on the problem.

Many characterizations, such as the one described in \cite{godsil2013algebraic} and the one that can be deduced from \cite{dehn1936kombinatorische} (see \cite{rosenstiehl1976gauss}), give a recognition algorithm that runs in time $O(n^2)$, where $n$ is the number of characters of the word. In \cite{rosenstiehl1984gauss} a linear time recognition algorithm, $O(n)$, is given. Both algorithms can be adapted to explicitly give a valid embedding of a curve with an added $O(n)$ time.

\subsection{Main results}

We present another characterization of Gauss codes, Theorem~\ref{t.charact_intro}, by giving an algorithm that, given a Gauss code $\omega$, finds in $O(n)$ time a combinatorial embedding of a closed curve whose traversed self-intersections give $\omega$.
The algorithm works for double occurrence paragraphs with no modifications. A double occurrence paragraph is a set of words that, if concatenated, form a double occurrence word. A paragraph with $k$ words is said to be a Gauss paragraph if there is an embedding of $k$ closed curves such that traversing the intersections of the $i$-th curve gives
the $i$-th word.

The characterization is based on the properties of the Seifert cycles of a closed curve embedded on the plane. The Seifert cycles of a knot were introduced by Seifert \cite{seifert1935geschlecht,adams2004knot} in order to construct a surface whose boundary is a given knot. As an embedded closed curved can be interpreted as the plane shadow of a knot (resp. a link), the Seifert cycles can be defined for a plane embedding of a closed curve (resp. curves). Moreover, the definition of a Seifert cycle can be further extended to double occurrence words (resp. paragraphs) (see Definition~\ref{d.seif_cycles}).

The characterization algorithm can be summarized as follows.
 From the word (resp. paragraph) $\omega$ we construct $\text{Sf}(\omega)$ the Seifert map of $\omega$. The Seifert map is multigraph, possibly containing loops, $2$-cell embedded on an orientable surface (see Definition~\ref{d.seif_map}). If $\omega$ is a Gauss code (resp. paragraph), then a suitable orientation $\sigma$ of the edges of $\text{Sf}(\omega)$ allows
 to find a medial-like oriented $4$-regular plane map $\text{Mv}(\text{Sf}(\omega),\sigma)$ (see Definition~\ref{d.med_or}) which gives the combinatorial plane embedding of $\omega$. The embedded closed curves are given by the closed walks in $\text{Mv}(\text{Sf}(\omega),\sigma)$ where the next edge of the walk is the ``straight ahead'' edge. The algorithm $\text{Alg}(\cdot)$, described in Section~\ref{sec.give_algo}, gives the sought after orientation $\sigma$. The maps
 $\text{Sf}(\cdot)$ and $\text{Mv}(\cdot)$ defined in Section~\ref{sec.operations} are invertible of each other.

\begin{theorem}\label{t.charact_intro}
Let $\omega$ be a double occurrence paragraph. If $\omega$ is a Gauss paragraph, then the procedure $\text{Alg}(\text{Sf}(\omega))$ gives an orientation $\sigma$ on the edges of $\text{Sf}(\omega)$, the Seifert map of $\omega$, such that $\text{Mv}(\text{Sf}(\omega),\sigma)$, the vertex-medial map of the oriented map $(\text{Sf}(\omega),\sigma)$,  is a combinatorial plane embedding that realizes $\omega$.
\end{theorem}


Moreover, Proposition~\ref{prop.give_all_embeddings} characterize all the plane embeddings of a Gauss paragraph. In particular,
the algorithm can be modified to identify and give all possible plane embeddings of the paragraph (see Proposition~\ref{prop.give_all_embeddings}). In such case, the running time is proportional to the number of characters times the number of possible plane embeddings of the paragraph.

In \cite{lins1987gauss} Lins, Richtfr and Shank, considered $2$-cell embedding of a closed curve in different compact surfaces (not necessarily orientable) where the graph of the faces is $2$-colorable (the vertices are the faces, and two faces are connected if they are separated by an edge, which is the part of the curve connecting two consecutive intersection points). In this work
the closed curves $2$-cell embedded in orientable surfaces are not necessarily two-colorable.

\subsection{Organization}

Section~\ref{sec.prev} is devoted to introduce the notions and operations on maps that are used in the subsequent sections. 
In Section~\ref{sec.char} we describe the algorithm $\text{Alg}(\cdot)$ that gives an orientation to the edges of the map, and its properties are discussed in 
Section~\ref{sec.alg_cor}. 
 Finally, Section~\ref{sec.main_proof} is devoted to show the main result,
Theorem~\ref{t.characterization}, from which Theorem~\ref{t.charact_intro} can be deduced.

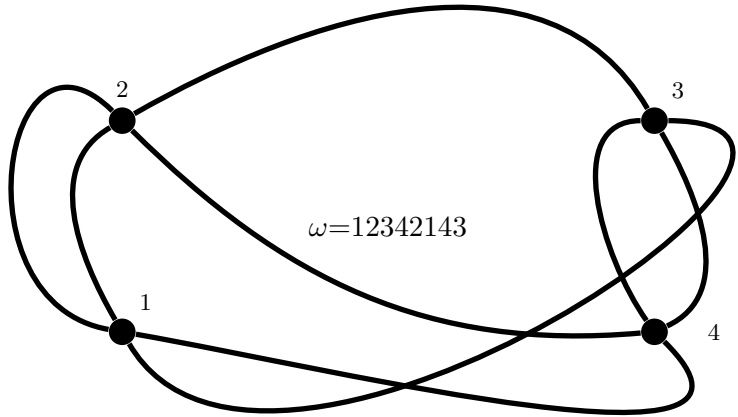
\begin{figure}[ht]
	\centering
		\begin{tikzpicture}[scale=0.7]
		\tikzstyle{vertex}=[circle,fill=black!100,minimum size=10pt,inner sep=0pt]
	\useasboundingbox (-7.3,-4.5) rectangle (210pt,130pt);
		
		\node[vertex,label=60:{\footnotesize{$1$}}] (c0) at (-5,-2) {};
		\node[vertex,label={\footnotesize{$2$}}] (c1) at (-5,2) {};
		\node[vertex,label=60:{\footnotesize{$3$}}] (c2) at (5,2) {};
		\node[vertex,label=right:{~~~\footnotesize{$4$}}] (c3) at (5,-2) {};

		
		\draw[line width=2pt] (c0) to [out=120, in=210] (c1);
		\draw[line width=2pt] (c1) to [out=30, in=120] (c2);
		\draw[line width=2pt] (c2) to [out=300, in=25] (c3);
		\draw[line width=2pt] (c3) to [out=185,in=315] (c1);
		\draw[line width=2pt] (c1) to [out=135,in=170,looseness=1.8] (c0);
		\draw[line width=2pt] (c0) to [out=350, in=315] (c3);
		\draw[line width=2pt] (c3) to [out=125, in=180] (c2);
		\draw[line width=2pt] (c2) to [out=0,in=300,looseness=1.2] (c0);
		
		\node (word) at (0,0) {$\omega$=12342143};
		\end{tikzpicture}
	\caption{
	An embeddable closed curve and its related word}\label{fig.5}
\end{figure}

\section{Preliminaries}\label{sec.prev}

\subsection{On words}

\begin{defn}[Double occurrence word, paragraphs]
	A \emph{word} is a sequence of characters of an alphabet.
	A \emph{cyclic word} is a class of words in the quotient space given by the ``being cyclically equivalent'' relation:
	\[\omega\sim u
	\iff |\omega|=|u| \text{ and } \exists i \text{ such that }\forall j\;
	\omega_{j+i\text{ mod } |\omega|}=u_{j \text{ mod } |\omega|}. 
	\]
	A \emph{double occurrence word} is a finite cyclic word in which every character appearing in the word appears exactly twice.	
	A \emph{double occurrence paragraph} is a set of finite words whose concatenation forms a double occurrence word.
\end{defn}

\begin{defn}[Gauss code]\label{d.gauss_code}
	A double occurrence word $\omega$ is said to be a \emph{Gauss code}, or a \emph{realizable word}
	if there exists an embedding $\epsilon$ of a closed curve in the sphere, with finitely many proper self-intersections (no tangent points, each crossing point is crossed exactly twice) with the following property. There is an assignment of the characters of $\omega$ to the crossing points of $\epsilon$ such that, traversing $\epsilon$ in a certain direction, traverses the characters in the cyclic order as in $\omega$.
	
	Similarly, a double occurrence paragraph $\omega'$ with $k$ words is said to be a \emph{Gauss paragraph}, or a \emph{realizable paragraph} if there is an embedding of $k$ closed curves on the sphere, each intersection point being crossed exactly twice non-tangently, and where the traversed intersection points in the $i$-th curve gives the $i$-th word.
%
%
	
\end{defn}

\subsection{On maps}
A map is a graph $2$-cell-embedded on a surface ($2$-dimensional compact manifold), where the vertices are points of the surface and the edges are curves between vertices possibly intersecting other edges only at its endpoints. All the surfaces considered in this work are orientable, and all the maps induce orientable surfaces. Thanks to Edmonds' theorem \cite{edmonds1960combinatorial}, a combinatorial description of a (oriented) map may be defined as follows.


\begin{defn}[Map, see \cite{lando2013graphs}]\label{d.map}
	A \emph{map} $M$ is a triple $(D;\tau,\alpha)$, where:
	
$D$ is the set of \emph{darts} of even size.

$\alpha:D\to D$ is an involution ($\alpha^2$ is the identity) with no fix points.

$\tau:D\to D$ is a permutation.
\end{defn}

		The \emph{edge set} of a map $M$, $E(M)$, is conformed by the pairs of darts $\{d,\alpha(d)\}$.
		The \emph{vertex set} of $M$, $V(M)$, is conformed by the cycles of the permutation $\tau$.
		The graph formed by the edge set and the vertex set of $M$ is known as the \emph{underlying graph} of $M$.	The \emph{face set} of $M$ is conformed by cycles of the permutation $\tau\alpha$ on the darts.
		The \emph{genus} of (an oriented map) $M$, $g(M)$, is given by Euler's relation $|V(M)|-|E(M)|+|F(M)|=2k(M)-2g(M)$, where $k(M)$ is the number of connected components of the underlying graph of $M$.
 		If a positive orientation is given to the surface, then the darts are drawn around the vertex (a point on the surface) in counterclockwise order (positive with respect to the orientation) following $\tau$.

\begin{defn}[Orientation on a map]\label{d.orient}
	An \emph{orientation} $\sigma$ on a map $M=(D;\tau,\alpha)$ is a function
	$\sigma: D\to \{-1,+1\}$ such that $\sigma(d)\sigma(\alpha(d))=-1$.
	A pair $(M,\sigma)$ 
	defines an \emph{oriented map}.
	\end{defn}

	An orientation of the map defines an orientation of the underlying graph where the edge $\{d,\alpha(d)\}$ is oriented from the dart in $\{d,\alpha(d)\}$ having image $-1$ (tail of the edge) to the dart in $\{d,\alpha(d)\}$ whose image is $+1$ (tip of the edge).

\begin{defn}[$2/2$-map]
	An oriented map $(M,\sigma)=((D;\tau,\alpha),\sigma)$ is said to be a \emph{$2/2$-map} if the orientation $\sigma$ is such 
		each vertex has two darts oriented $+1$, two darts oriented $-1$,
		and all the $+1$ darts are consecutive in the permutation cycles of $\tau$.
\end{defn}

\begin{observation}\label{obs.map_to_double_para}
A $2/2$-map $M$ induces an embedding of closed curves in a surface of genus $g(M)$ as follows. Each curve is given by the closed circuit obtained as a sequence of directed edges $(e_1,\ldots,e_k,e_1)$, where $e_{i+1}$ is the ``straight ahead'' edge leaving the vertex at the tip of $e_i$ (see Figure~\ref{fig.2}, where each vertex is character corresponds to a vertex, and the cyclic orientation of the map is given by the edges in counterclockwise cyclic order, the straight ahead orientation gives $\omega$). That is to say, the tail of $e_{i+1}$ and the tip of $e_i$ belong to the same permutation cycle, yet they are not consecutive one another. 
The vertices are the intersection points of the curves. If each vertex is given a different character, then a $2/2$-map induces a double occurrence paragraph using as words the sequence of vertices induced by the closed walks.
\end{observation}

%
%
%
%
%


\begin{defn}[$2/2$-premap]\label{}
	A $(P,\sigma)=((D;\overline{\tau},\alpha),\sigma)$ is said to be a \emph{$2/2$-premap with orientation $\sigma$} if $\overline{\tau}$ is a set of permutations on $D$ such that,
	for each $\tau$ in $\overline{\tau}$, $((D;\tau,\alpha),\sigma)$ is a $2/2$-map. Moreover, if
$\tau\in \overline{\tau}$ with permutation cycles $c_1,\ldots,c_l$, then
\[
\overline{\tau}=\left\{ \tau=c_1^{\epsilon_1}\; \cdots\;c_l^{\epsilon_l}\text { with }\epsilon_i\in \{+1,-1\} \text{ for all } i\in[l] \right\}.
\]
	\end{defn}

 In particular, if $(d_1\; \,d_2\;\,d_3\;\,d_4)$ is a cycle in $\tau\in \overline{\tau}$, with $d_1,d_2\in\sigma^{-1}(+1)\subset D,\; d_3,d_4\in\sigma^{-1}(-1)$, then any other permutation in $\overline{\tau}$ contains either $(d_2\;\,d_1\;\,d_4\;\,d_3)$ or $(d_1\;\,d_2\;\,d_3\;\,d_4)$ as a cycle. 


\begin{observation}\label{obs.algo_algo}
	Let $(P,\sigma)=((D;\overline{\tau},\alpha),\sigma)$ be an $2/2$-premap.
	\begin{enumerate}
		\item $(P,\sigma)$ induces a graph $G=(V,E)$ as follows. The set of vertices $V$ are the permutation cycles of any $\tau\in\overline{\tau}$. The mapping $\nu:D\to V$ assigns, to the dart $d$, the vertex corresponding to the permutation cycle containing $d$. The edges correspond to the pairs $\{d,\alpha(d)\}_{d\in D}$ and join the vertices $\nu(d)$ and $\nu(\alpha(d))$. The graph $G$ only depends on the set of maps $\overline{\tau}$ and not on the particular chosen map $\tau\in \overline{\tau}$.
		
				\item \label{prop.orient_3} 
				In a $2/2$-map, every dart $d$ have the property that $\sigma(\tau(d))\sigma(\tau^{-1}(d))=-1$.  Thus any dart $d$
				has a unique neighbour with the same orientation and a unique neighbor with the opposite orientation.
				These unique darts are the same for any $2/2$-map representative of $(P,\sigma)$.
				 Hence, we may define the  map $\rho$ for $2/2$-premaps as follows. Given any permutation $\tau\in \overline{\tau}$ of the premap and a dart $d$,
		\[
		\begin{cases}
		\rho(d)=\alpha(\tau(d)) &\text{ if }\;
		\sigma(d)=+1 \text{ and } \sigma(\tau(d))=-1\\
		
		\rho(d)=\tau(\alpha(d)) &\text{ if }\;
		\sigma(d)=-1 \text{ and } \sigma(\tau(\alpha(d)))=-1\\

		\rho(d)=\alpha(\tau^{-1}(d)) &\text{ if }\;
		\sigma(d)=+1 \text{ and } \sigma(\tau(d))=+1\\
		\rho(d)=\tau^{-1}(\alpha(d)) &\text{ if }\;		\sigma(d)=-1 \text{ and } \sigma(\tau(\alpha(d)))=+1\\
		\end{cases}
		\]		
		Observe that $\sigma(d)\sigma(\rho(d))=1$.
		In Figure~\ref{fig.2}, $\rho$ of the dart at $2$ from the edge $(2,3)$ is the dart at $3$ from the edge $(3,4)$ since we are in the fourth case.
%


%
%
%
%
%

		
		\item\label{obs.algo} For any $2/2$-map $\tau\in \overline{\tau}$, and any dart $d$, the dart in $\nu(d)$ with opposite sign to $d$ and not consecutive to $d$ (in the permutation cycle of $\nu(d)$, the following ``straight ahead'' edge) is the same (does not depend on $\tau$, but on $\overline{\tau}$). Therefore the double occurrence paragraph defined using Observation~\ref{obs.map_to_double_para}
		is independent of $\tau$, and depends only on $\overline{\tau}$.
		
		Moreover, a $2/2$-premap codifies each of the possible embeddings (where each point of intersection of the curves is non-tangent and traversed exactly twice) of the paragraph (defined by $\overline{\tau}$). 
	\end{enumerate}
\end{observation}

\begin{observation}[$2/2$-premap from a paragraph]\label{obs.orient_premap_from_word}
	The double occurrence paragraph $\omega$ induces a $2/2$-premap $((D;\overline{\tau},\alpha),\sigma)$ as follows.
	The set of darts $D$ is $\{\omega_{(i,+)},\omega_{(i,-)}\}_{i\in [1,|\omega|]}$, two for each copy of each character in $\omega$.
	
	$\sigma(\omega_{(i,-)})=-1$ and $\sigma(\omega_{(i,+)})=+1$ for each $i\in[1,|\omega|]$.
	
	If $\omega_i$ and $\omega_j$, are the two copies of the same character in $\omega$, then for each $\tau\in \overline{\tau}$ either
	$(\omega_{(i,+)}\,\;\omega_{(j,+)}\,\;\omega_{(i,-)}\,\;\omega_{(j,-)})$ or
	$(\omega_{(j,+)}\;\,\omega_{(i,+)}\,\;\omega_{(j,-)}\,\;\omega_{(i,-)})$
	is a permutation cycle of $\tau$.

	
	$\alpha(\omega_{(i,+)})=\omega_{(p(i),-)}$ and $\alpha(\omega_{(i,-)})=\omega_{(n(i),\cdot)}$, where $\omega_{(p(i),\cdot)}$ and $\omega_{(n(i),-)}$ denote, respectively, the previous and next characters of $\omega_{i}$ in its word.
\end{observation}

In Figure~\ref{fig.2} (and Figure~\ref{fig.5}) we see one of the maps in the $2/2$-premap corresponding to the word $\omega=12342143$ in which, for instance $\{\omega_{(5,+)},\omega_{(5,-)}\}$ corresponds to two darts associated to the second copy of the character $2$, and, together with the two darts $\{\omega_{(2,+)},\omega_{(2,-)}\}$, they generate the permutation cycle
$(\omega_{(2,+)},\omega_{(5,+)},\omega_{(2,-)},\omega_{(5,-)})=((1\to2),(4\to2),(2\to3),(2\to 1))$ which corresponds to the vertex labeled $2$ in the figure.

%
%
%

\begin{defn}[Seifert cycles of a $2/2$-(pre)map]\label{d.seif_cycles}
	Let $(M,\sigma)=((D;\tau,\alpha),\sigma)$ be a $2/2$-map.
Then 
	\begin{itemize}
		\item The \emph{Seifert cycle} of the dart $d$, $S(d)$, is defined as the (cyclically ordered) sequence of unrepeated, equally-oriented darts \[(d\,\;\rho(d)\,\;\rho^2(d)\,\;\cdots\,\;\rho^{-1}(d)).\]
		\item The \emph{Seifert cycles of $(M,\sigma)$} is the set $\mathcal{S}(M,\sigma)=\{S(d)\}_{d\in D(M),\sigma(d)=+1}$, where two sequences are the same if they are cyclically equivalent.
		\end{itemize}
	Seifert cycles are similarly defined for $2/2$-premaps $(P,\sigma)$ as $\rho$ is well defined for $2/2$-premaps.	
\end{defn}

\begin{observation}
	If the $2/2$-map is plane, then the Seifert cycles correspond to the Seifert cycles of the shadow of the knot (or link) it induces.
\end{observation}

\subsection{Operations on oriented maps}\label{sec.operations}


\begin{defn}[Seifert (un)oriented map from a $2/2$-(pre)map)]\label{d.seif_map}
	Let $(P,\sigma)=((D;\overline{\tau},\alpha),\sigma)$ be a $2/2$-premap.  The \emph{Seifert map} $\text{Sf}(P,\sigma)=M'=(D';\tau',\alpha')$ is defined as follows.
		
		The vertices of $M'$ are the Seifert cycles of $(P,\sigma)$, so $V(M')=\mathcal{S}(P,\sigma)$.
		
		The darts of $M'$ are the positively oriented darts of $P$. $D'=\sigma^{-1}(+1)\cap D$.
		
		
		For $d_1,d_2\in D'\subset D$, $\alpha'(d_1)=d_2$ if and only if $\nu_{(P,\sigma)}(d_1)=\nu_{(P,\sigma)}(d_2)$.

		
		For each  $(d_1,\ldots,d_k)\in V(M')$, then
		 $\tau'(d_i)=d_{i+1}$ with $i$ and $i+1$ being modulo $k$.


If $(M,\sigma)=((D;\tau,\alpha),\sigma)$ is a $2/2$-map with $\tau\in \overline{\tau}$, then the orientation $\sigma'$ of the \emph{Seifert oriented map} $\text{Sf}(M,\sigma)=(M',\sigma')=(\text{Sf}(P,\sigma),\sigma')$ is given by, for $d\in D'$,
\[
\begin{cases}
\sigma'(d)=+1 &\text{ if }\sigma(\tau(d))=-1 \\
\sigma'(d)=-1 &\text{ if }\sigma(\tau(d))=+1
\end{cases}\] 
\end{defn}

The Seifert cycles and the Seifert map of a paragraph are the Seifert cycles and the Seifert map of its associated $2/2$-premap.
Additionally, every map is the Seifert map of a double occurrence paragraph (see Proposition~\ref{prop.inv_op} and Observation~\ref{obs.algo_algo} point \ref{obs.algo}). See Figure~\ref{fig.2} and Figure~\ref{fig.2} for the example related to Figure~\ref{fig.5}.

\begin{observation}
	In the Seifert map, each vertex is associated with a Seifert cycle $c=(e_1,\ldots,e_k)$ and has, as attached edges, the edges $e_1,\ldots,e_k$ (the positive darts) corresponding to the cycle, and whose cyclic orientation is that of $c$. Two Seifert cycles (vertices) are connected by an edge in the Seifert map if they meet at a vertex of the $2/2$-premap. See Figure~\ref{fig.2} and Figure~\ref{fig.1}.
\end{observation}

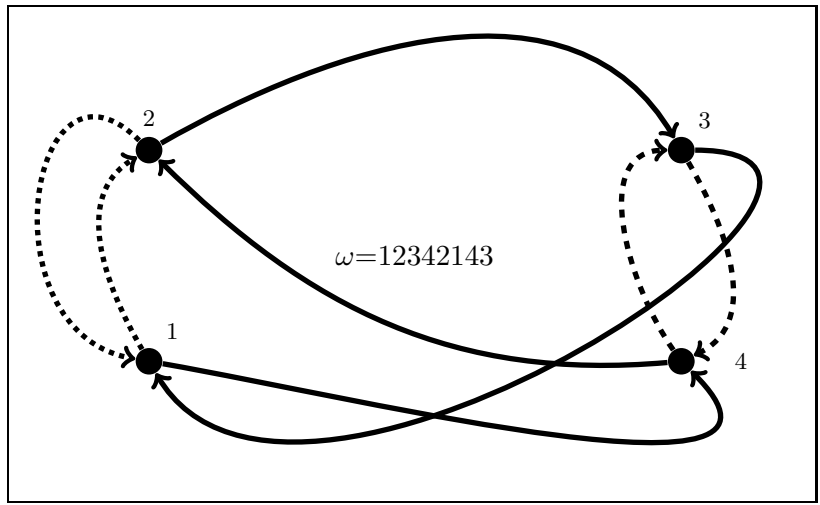
\begin{figure}[ht]
	\centering
	\fbox{
	\begin{tikzpicture}[scale=0.7, shorten >=0.0pt,->]
	\tikzstyle{vertex}=[circle,fill=black!100,minimum size=10pt,inner sep=0pt]
	\useasboundingbox (-7.3,-4.5) rectangle (210pt,130pt);

	\node[vertex,label=60:{\footnotesize{$1$}}] (c0) at (-5,-2) {};
	\node[vertex,label={\footnotesize{$2$}}] (c1) at (-5,2) {};
	\node[vertex,label=60:{\footnotesize{$3$}}] (c2) at (5,2) {};
	\node[vertex,label=right:{~~~\footnotesize{$4$}}] (c3) at (5,-2) {};

	
	\draw[dotted,line width=2pt,] (c0) to [out=120, in=210] (c1);
	\draw[line width=2pt] (c1) to [out=30, in=120] (c2);
	\draw[dashed,line width=2pt] (c2) to [out=300, in=25] (c3);
	\draw[line width=2pt] (c3) to [out=185,in=315] (c1);
	\draw[dotted,line width=2pt] (c1) to [out=135,in=170,looseness=1.8] (c0);
	\draw[line width=2pt] (c0) to [out=350, in=315] (c3);
	\draw[dashed,line width=2pt] (c3) to [out=125, in=180] (c2);
	\draw[line width=2pt] (c2) to [out=0,in=300,looseness=1.2] (c0);
	
	\node (word) at (0,0) {$\omega$=12342143};
	\end{tikzpicture}}
	\caption{Seifert cycles}\label{fig.2}
\end{figure}

	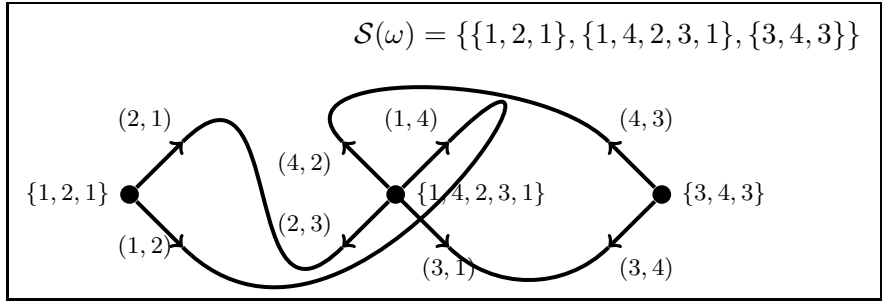
\begin{figure}[ht]
		\centering
		\fbox{\begin{tikzpicture}[scale=0.7]
		\tikzstyle{vertex}=[circle,fill=black!100,minimum size=7pt,inner sep=0pt]
		
		\node (seif) at (4,3) {$\mathcal{S}(\omega)=\{\{1,2,1\},\{1,4,2,3,1\},\{3,4,3\}\}$};

		\node[vertex,label=0:{\footnotesize{$\{1,4,2,3,1\}$}}] (v0) at (0,0) {};
		\node[vertex,label=0:{\footnotesize{$\{3,4,3\}$}}] (v1) at (5,0) {};
		\node[vertex,label=180:{\footnotesize{$\{1,2,1\}$}}] (v2) at (-5,0) {};
		
		
		\coordinate[label=270:{\footnotesize{$(3,1)$}}] (d1) at (1,-1);
		\coordinate[label=120:{\footnotesize{$(1,4)$}}] (d2) at (1,1);
		\coordinate[label=230:{\footnotesize{$(4,2)$}}] (d3) at (-1,1);
		\coordinate[label=120:{\footnotesize{$(2,3)$}}] (d4) at (-1,-1);
		
		\coordinate[label=60:{\footnotesize{$(4,3)$}}] (d7) at (4,1);
		\coordinate[label=315:{\footnotesize{$(3,4)$}}] (d8) at (4,-1);
		
		\coordinate[label=180:{\footnotesize{$(1,2)$}}] (d5) at (-4,-1);
		\coordinate[label=120:{\footnotesize{$(2,1)$}}] (d6) at (-4,1);
		

		\useasboundingbox (0,0) rectangle (10pt,10pt);
		
		\draw[->,line width=1.5pt] (0,0) -- (1,-1);
		\draw[->,line width=1.5pt] (v0) -- (d2);
		\draw[->,line width=1.5pt] (v0) -- (d3);
		\draw[->,line width=1.5pt] (v0) -- (d4);

		\draw[->,line width=1.5pt] (v1) -- (d7);
		\draw[->,line width=1.5pt] (v1) --(d8);
		
		\draw[->,line width=1.5pt] (v2) -- (d5);
		\draw[->,line width=1.5pt] (v2) -- (d6);
		
		\draw[-,line width=1.5pt] (d1) to [out=315,in=225] (d8);
		\draw[-,line width=1.5pt] (d2) to [out=45,in=315,looseness=2.2] (d5);
		\draw[-,line width=1.5pt] (d3) to [out=135,in=135] (d7);
		\draw[-,line width=1.5pt] (d4) to [out=225,in=45,looseness=2] (d6);
		\end{tikzpicture}}
		\caption{Seifert map of $\omega=12342143$.}\label{fig.1}
	\end{figure}

\begin{defn}[Vertex-medial (pre)map from an (un)oriented map]\label{d.med_or}
	Let $M=(D;\tau,\alpha)$ be a map.
	Then the \emph{vertex-medial premap of $M$}, $\text{Mv}(M)=(P,\sigma')=((D';\overline{\tau}',\alpha'),\sigma')$ is a $2/2$-premap constructed as follows.
		
		The vertex set of $P$ is the edge set of $M$: $V(P)= \{\{d,\sigma(d)\}\}_{d\in D}$.
		
		The set of darts $D'$ is the union of $D\times \{+\}$ and $D\times \{-\}$. (The elements $(d,+)\in D\times\{+\}$ and $(d,-)\in D\times\{-\}$ are denoted by $d_{+}$ and $d_{-}$ respectively.)
		
		
		The orientation $\sigma'$ is given by $\sigma'(d_+)=+1$ and $\sigma'(d_-)=-1$
		
		 
		 $\alpha'$ is given by
		$\alpha'(d_{-})=\tau(d)_{+}$ and 
		$\alpha'(d_{+})=\tau^{-1}(d)_{-}$.
%
		
		
%
		Each $\tau'\in \overline{\tau}'$ is formed by the $|D|/2$ permutations cycles $\{c_{\{d,\alpha(d)\}}\}_{\{d,\alpha(d)\}\in V(P)}$ with
		\[c_{\{d,\alpha(d)\}}\in\left\{(
		\alpha(d)_{+}\,\;d_{+}\,\;d_{-}\,\;\alpha(d)_{-}),\,\;(d_{+}\,\;
		\alpha(d)_{+}\,\;\alpha(d)_{-}\,\;d_{-})\right\}.\]


If the map $M$ is oriented by $\sigma$, then the map $\tau'\in \overline{\tau}'$ of the \emph{vertex-medial map} 
$\text{Mv}(M,\sigma)=((D';\tau',\alpha'),\sigma')$ is given, for $d\in D$, by
			\[
			\begin{cases}
			\text{ If }\sigma(d)=-1, \sigma(\alpha(d))=+1 \text{ then }&
						\begin{cases}
\tau'(d_{+})=\alpha(d)_{+}, \; \tau'(\alpha(d)_{+})=\alpha(d)_{-},\\ \tau'(\alpha(d)_{-})=d_{-},\;
\tau'(d_{-})=d_{+}
\end{cases}\\

			\text{ If }\sigma(d)=+1,\sigma(\alpha(d))=-1 \text{ then }&
						\begin{cases}
			\tau'(d_{+})=d_{-},\;
			\tau'(d_{-})=\alpha(d)_{-}\\
			\tau'(\alpha(d)_{-})=\alpha(d)_{+},\;
			\tau'(\alpha(d)_{+})=d_{+}
			\end{cases}\\
						\end{cases}.
			\]
			
\end{defn}

\begin{observation}\label{obs.prop_vm}
	In the vertex-medial (pre)map of $M$, each edge in $M$ is a vertex in $\text{Mv}(M)$. Each pair of consecutive darts at a vertex constitutes an edge. For each dart $d$, the previous edge (consecutive pair of darts $(\text{``previous to $d$''}, d)\;$) and next edge (the consecutive pair of darts $(\text{``next to $d$''},d)\;$) at the vertex $\nu(d)$ in $M$ are consecutive in the vertex $\{d,\alpha(d)\}$ in $\text{Mv}(M)$, thus in particular belong to the same face. If $(d_1,\ldots,d_k)$ is a permutation cycle corresponding to a vertex in $M$, then the edges $(d_1,d_2),(d_2,d_3),\ldots,(d_k,d_1)$ corresponding to consecutive pairs of darts in $M$ constitute a closed circuit in $\text{Mv}(M)$.
\end{observation}

\begin{remark}
	The medial graph construction of the map $M$ involves creating a $4$-regular map where the edges of $M$ become the vertices, where $e_1$ is adjacent to $e_2$ if $e_1$ is the previous (or next) edge to $e_2$ along a face rotation (e.g. \cite{archdeacon1992medial}), and the cycle rotation is determined by the rotation of the neighbouring edges along the (positive) orientation of the surface.
In particular, 
	the contiguity of two darts in the vertex rotations are determined by the cyclic rotations of the faces of $M$. In contrast, the contiguity in Definition~\ref{d.med_or} is determined by the vertex rotations of $M$.
	For instance, the medial map of a loop is a vertex with two loops on the plane. However, the vertex-medial map of a loop is a vertex with two loops embedded in the torus. See Figure~\ref{fig.6} for an example.
%
%
%
\end{remark}

	\begin{figure}[ht]
		\centering
		\fbox{\begin{tikzpicture}[scale=0.7]
		\tikzstyle{vertex}=[circle,fill=black!100,minimum size=7pt,inner sep=0pt]
		

		\node[vertex,label=0:{\footnotesize{$\vdots$}}] (n0) at (-2,0) {};
		\node[vertex,label=180:{\footnotesize{$\vdots$}}] (n1) at (-5,0) {};
		\draw[-,line width=1pt] (n0) -- (n1);

		
		
		\coordinate[label=0:{\footnotesize{$d_1$}}] (d1) at (-6,1);
		\coordinate[label=110:{\footnotesize{$d_2$}}] (d2) at (-4,0);
		\coordinate[label=0:{\footnotesize{$d_3$}}] (d3) at (-6,-1);
		\coordinate[label=70:{\footnotesize{$d_4$}}] (d4) at (-3,0);
		
		\coordinate[label=200:{\footnotesize{$d_5$}}] (d5) at (-1,-1);
		\coordinate[label=120:{\footnotesize{$d_6$}}] (d6) at (-1,1);
		
		\filldraw  ([xshift=0pt,yshift=-4pt]d2) rectangle ++(2pt,8pt);
		\filldraw  ([xshift=0pt,yshift=-4pt]d4) rectangle ++(2pt,8pt);
		
		\draw[-,line width=1pt] (n0) -- (d6);
		\draw[-,line width=1pt] (n0) -- (d5);
		\draw[-,line width=1pt] (n1) -- (d1);
		\draw[-,line width=1pt] (n1) -- (d3);		

		\node[vertex,label=0:{\footnotesize{$(d_{4+},d_{2+},d_{2-},d_{4-})$}}] (n2) at (3,2) {};
		
		\node[vertex,label=0:{\footnotesize{$(d_{2+},d_{4+},d_{4-},d_{2-})$}}] (n3) at (3,-2) {};

		\coordinate[label=180:{\footnotesize{$\{d_{2-},d_{1+}\}$}}] (d1') at (2,3);
		\coordinate[label=0:{\footnotesize{$\{d_{4+},d_{6-}\}$}}] (d3') at (4,1);
		\coordinate[label=0:{\footnotesize{$\{d_{2+},d_{3-}\}$}}] (d6') at (4,3);
		\coordinate[label=270:{\footnotesize{$\{d_{4-},d_{5+}\}$}}] (d5') at (2,1);
		
		\draw[->] (n2)-- (d1');
		\draw[<-] (n2)-- (d3');
		\draw[<-] (n2)-- (d6');
		\draw[->] (n2)-- (d5');

		\coordinate[label=90:{\footnotesize{$\{d_{2-},d_{1+}\}$}}] (d1'') at (2,-1);
		\coordinate[label=0:{\footnotesize{$\{d_{4+},d_{6-}\}$}}] (d3'') at (4,-3);
		\coordinate[label=180:{\footnotesize{$\{d_{2+},d_{3-}\}$}}] (d6'') at (2,-3);
		\coordinate[label=0:{\footnotesize{$\{d_{4-},d_{5+}\}$}}] (d5'') at (4,-1);
		
		\draw[->] (n3)-- (d1'');
		\draw[<-] (n3)-- (d3'');
		\draw[<-] (n3)-- (d6'');
		\draw[->] (n3)-- (d5'');
		
		\draw[->,line width=2pt] (-0.5,0.5) -- (1.5,1.5);
		\coordinate[label=90:{\tiny{$\sigma(d_2)=+1,\sigma(d_4)=-1$}}]
		(lab_line_1) at (10,2.5);
		\draw[->,line width=2pt] (-0.5,-0.5) -- (1.5,-1.5);
				\coordinate[label=90:{\tiny{$\sigma(d_2)=-1,\sigma(d_4)=+1$}}]
		(lab_line_1) at (10,-3);
		
		\end{tikzpicture}}
		\caption{Vertex-medial premap}\label{fig.6}
	\end{figure}
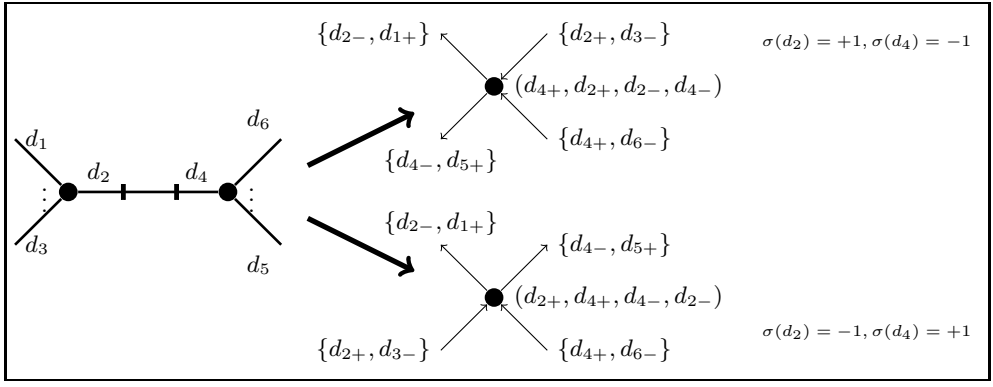

%

\begin{proposition}[Invertible operations]\label{prop.inv_op}
	$\text{Sf}(\cdot)$ and $\text{Mv}(\cdot)$ are invertible operations.
	\begin{enumerate}
	\item\label{p.inv.1}
	If $M$ is a map, then $\text{Sf}(\text{Mv}(M))=M$. 
	
	If $(M,\sigma)$ is an oriented map, then $\text{Sf}(\text{Mv}(M,\sigma))=(M,\sigma)$.
	\item\label{p.inv.2} If $(P,\sigma)$ is a $2/2$-premap, then $\text{Mv}(\text{Sf}(P,\sigma))=(P,\sigma)$. 
	
	If $(M,\sigma)$ is a $2/2$-map, then
	$\text{Mv}(\text{Sf}(M,\sigma))=(M,\sigma)$.
	
	\item \label{p.bij-orient-map}
	There is a bijection between the orientations of the map $M$ and the maps in the $2/2$-premap $\text{Mv}(M)$.
	\end{enumerate}
\end{proposition}

\begin{proof}[Proof of Proposition~\ref{prop.inv_op}]
	Let us show part \ref{p.inv.1}. Let  $M=(D;\tau,\alpha)$ be a map. Let
		$\text{Mv}(M)=((D';\overline{\tau}',\alpha'),\sigma')$ and
		$\text{Sf}(\text{Mv}(M))=(D'';\tau'',\alpha'')$ denote these two images. We shall see that that $D''=D$, $\tau''=\tau$, and $\alpha''=\alpha$.
		
		We have $D'=D\times \{+\}\cup D\times \{-\}$ while the set of darts $D''$ is the set of positively oriented darts from $D'$, hence $D''=D$.
		
%
%

		Let $d\in D$ be a positively oriented dart. Then
		$\alpha'(d_{-})=\tau(d)_{+}$ and $\alpha'(d_{+})=\tau^{-1}(d)_{-}$, any $\tau'\in \overline{\tau}'$ contains either
		$(\alpha(d)_+\,\;d_+\,\;d_{-}\,\;\alpha(d)_-)$  or  $(d_+\,\;\alpha(d)_+\,\:\alpha(d)_{-}\,\;d_-)$
		as a permutation cycle. Moreover, $\sigma'(d_{+})=+1$ and $\sigma'(d_{-})=-1$. Since, $\alpha(d)_+$ is the other positively signed dart in the same permutation cycle as $d_+$, then  $\alpha''(d_+)=\alpha(d)_+$, which implies $\alpha=\alpha''$. 
		Also, 
		\[\tau''(d_{+})=\rho'(d_{+})=
		\begin{cases}
\alpha'(\tau'(d_{+}))=\alpha'(d_{-})=\tau(d)_{+} &\text{ if }\;
\sigma'(d_{+})=+1 \text{ and } \sigma'(\tau'(d_{+}))=-1\\

\alpha'(\tau'^{-1}(d_{+}))=\alpha'(d_{-})=\tau(d)_{+} &\text{ if }\;
\sigma'(d_{+})=+1 \text{ and } \sigma'(\tau'(d_{+}))=+1\\
\end{cases}
\]		which shows that $\tau=\tau''$.

To show the second part of \ref{p.inv.1}, let $\sigma$ be an orientation of the map $M$. Then, if $\sigma(d)=+1$, $\tau'$ contains the cycle
\[
(\alpha(d)_{+}\,\;d_+\,\;d_{-}\,\;\alpha(d)_{-})
\]
which implies that $\sigma'(\tau'(d_{+}))=\sigma'(d_{-})=-1$, hence giving back $\sigma''(d)=+1$. A similar argument shows the case for  $\sigma(d)=-1$.

Now let us show part \ref{p.inv.2}. Let $(P,\sigma)=((D;\overline{\tau},\alpha),\sigma)$ be a $2/2$-premap, let $\text{Sf}(P,\sigma)=(D';\tau',\alpha')$, and let  $((D'';\overline{\tau}'',\alpha''),\sigma'')=\text{Mv}(\text{Sf}(P,\sigma))$ be another $2/2$-premap.

	$D'$ are the positively oriented darts of $D$ and $D''=D'\times\{+\}\cup D'\times \{-\}$, hence clearly $D''=D$.

	Given a positively oriented $d\in D$, let either $(d\,\; d' \,\; d_1\,\; d_2)$ or $(d'\,\;d\,\;d_2\,\;d_1)$ be the permutation cycle of any $\tau\in \overline{\tau}$ associated to $\nu(d)$, and where $d'$ denotes the other positive dart in the mentioned permutation cycle. Then $\alpha'(d)=d'$. Furthermore, by definition of $\tau'$, \[
	\tau'(d)=\rho(d)=
	\begin{cases}
	\alpha(\tau(d))=\alpha(d_2) &\text{ if } \sigma(d)=+1 \text{ and } \sigma(\tau(d))=-1\\
	\alpha(\tau ^{-1}(d))=\alpha(d_2) &\text{ if } \sigma(d)=+1 \text{ and } \sigma(\tau(d))=+1
	\end{cases}
	\]
	where $\tau$ denotes any element in $\overline{\tau}$ (recall that $\rho$ is well defined on $\overline{\tau}$).

	Since $\alpha'(d)=d'$, any $\tau''\in \overline{\tau}''$ is such that the cycle containing $d_{+}$ and $d_{-}$ is either
	$
	(d'_{+}\,\;d_+\,\;d_{-}\,\;d'_{-})$ or $(d_{+}\,\;d'_+\,\;d'_{-}\,\;d_{-})$.
	
	Observe that 
	\[
	\alpha''(d_{-})=\tau'(d)_{+}=	\begin{cases}
	\alpha(\tau(d))_{+}=\alpha(d_2)_{+} &\text{ if } \sigma(d)=+1 \text{ and } \sigma(\tau(d))=-1\\
	\alpha(\tau ^{-1}(d))_{+}=\alpha(d_2)_{+} &\text{ if } \sigma(d)=+1 \text{ and } \sigma(\tau(d))=+1
	\end{cases}.
	\]
	Therefore, $d_{-}\in D''$ has the same behavior as $d_2\in D$. In particular, if in the bijection $D''\leftrightarrow D$ we identify $d_2$ and $d_{-}$ then $\alpha''=\alpha$ and $\overline{\tau}''=\overline{\tau}$ as the permutations become the same (both behave in the same way).

	
	To show the second part of part \ref{p.inv.2}, let $\tau$ be a given map from  $\overline{\tau}$ such that $(d\,\;d'\,\; d_1\,\;d_2)$ is the cycle corresponding to the dart $d$ with $\sigma(d)=\sigma(d')=+1$ as before. Then
	the orientation on the Seifert map, $\sigma'$ is given by
	\[
	\sigma'(d)=-1
	\]
	as $\sigma(\tau(d))=\sigma(d')=-1$, which implies that $\tau''$ contains the permutation cycle
	\[
	(d_{+}\,\;\alpha'(d)_{+}\,\;\alpha'(d)_{-}\,\;d_{-})
	\]
	which is precisely 
	$(d_{+}\,\;d'_{+}\,\;d'_{-}\,\;d_{-})$ and coincides with $(d\,\;d'\,\;d_1\,\;d_2)$ in terms of behavior, as it has been previously seen.
%
%
%
%
%
%
%
%
%
%
%
%
%
%
%
%
%

%
%

Part~\ref{p.bij-orient-map} follows from the previous arguments by observing that an orientation on the edges in a map selects one of the two possible cyclic orientation of the darts around the vertex in its vertex-medial map, and vice-versa.
\end{proof}

\section{Characterizing Gauss codes among double occurrence words}\label{sec.char}

In this section we give an algorithmic characterization of Gauss paragraphs.

\subsection{Orientation and labeling algorithm}\label{sec.give_algo}

The algorithm described in this section, denoted by as $\text{Alg}(\cdot)$, has, as input, an unoriented map, and as output an oriented version of it. If the input of the algorithm is the Seifert map of a Gauss paragraph, then the vertex-medial map of the output gives a plane embedding of the paragraph.
The algorithm also checks two necessary conditions for the input to be the Seifert map of a Gauss paragraph (these conditions are not sufficient as Observation~\ref{obs.counter} shows).
The properties claimed for the algorithm are given in the statements Proposition~\ref{p.alg_correct} and Theorem~\ref{t.characterization}.

Before proceeding further, let us briefly comment some aspects of the algorithm. It divides the vertices into two categories, those with label $1$ and those with label $0$. If the map is the Seifert map of a doubly occurrence paragraph, the vertices represent Seifert cycles on the plane. With this interpretation, $1$ marks the cycle as counterclockwise (positive) oriented, while $0$ marks it as clockwise (negative) orientated.
The orientation of an edge of the map aims to asses the relative containment  relations in the plane between the two Seifert cycles meeting at an intersection point (see Proposition~\ref{prop.clar}).

\paragraph{Orientation algorithm $\text{Alg}(M)$.}

\begin{enumerate}
	\item Let $M$ be a map. 
	
	If $M$ is not connected, apply the algorithm to each connected component separately. 
	
	Give an absolute order to the vertices of $M$ and to the darts of $M$.		
	
	If $M$ is not bipartite, then $M$ is not the Seifert map of a Gauss paragraph.
	
	\item Select $v_0$, a non-cut vertex of $M$ ($v_0$ may contain loops).
	
	Root the map at $v_0$ and label $v_0$ with $1$.
	
	\item Consider the tree of $2$-connected blocks of $M$ \cite[Proposition~3.1.2]{diestel2005graph}.  The articulation points of the tree are the cut vertices of $M$.
	
	
	 Root the block of $v_0$ at $v_0$.

	The root of a $2$-connected block $B$ is the vertex in $B$ that disconnects the rest of the $2$-connected component from $v_0$, the root of the map $M$. 
	
	Each connected union of $2$-connected blocks is rooted, and its root is the root of the $2$-connected block closest to $v_0$.

	\item \label{step.def_sv}  Given $v$ a vertex in $M$, let $B_0',B_1',\ldots, B_k'$ be the connected components in $M\setminus \{v\}$, and $L_1,\ldots, L_r$ be the loops attached to $v$. 
	For each $i\in[0,k]$, let $B_i$ be the graph induced by $V(B_i')\cup \{v\}$. $B_0$ contains the vertex $v_0$.
	
 Let $d_1,\ldots,d_t$ be the cyclically ordered darts of $M$ attached to $v$. Let $S_v$ be the graph with vertex set $\{d_1,\;\ldots,d_t,\; B_0,B_1,\;\ldots,B_k,\;L_1,\;\ldots,L_r\}$. Join $B_i$ with the vertices $d_{i_1},\ldots,d_{i_s}$ if $v$ is connected to (the rest of) $B_i$ using the edges corresponding to $d_{i_1},\ldots,d_{i_s}$ in $M$. Join $d_{i}$ with $d_{i+1}$, indices modulo $t$ ($d_1,\ldots,d_t$ form a cycle). Join $L_i$ with the two darts $d_j,\alpha(d_j)$ in $v$ forming $L_i$.



	

 For every cut vertex $v$ or a vertex containing loops, (so $k\geq 1$ or $r\geq 1$) or $v_0$, we decide whether $B\in\{B_0,\linebreak[0] B_1, \ldots, B_k,\linebreak[0] L_1,\;\ldots,\;L_r\}$ is to the \emph{inside} of $v$ or to the \emph{outside} of $v$ as follows.
	\begin{enumerate}

		\item\label{step.planar} 
		
		If $S_v$ is planar, let $M_v$ be a plane embedding of $S_v$. The cycle $(d_1,\ldots,d_t)$ determines two faces in $M_v$, the \emph{outer} and the \emph{inner} face. A vertex is said to be to the \emph{inside} of $v$ if it is located in the inner face, and to the \emph{outside} of $v$ if it is located in the outer face. 
		
		If $v\neq v_0$, $B_0$ is in the outer face of $M_v$.
		
		If $v=v_0$, $B_0$ is in the inner face of $M_{v_0}$.

		
		\item\label{step.non-planar} If $S_v$ is not planar, then $M$ is not the Seifert map of a Gauss paragraph.
		
		Let $D_v$ be a plane drawing of $S_v$ with minimal number of crossings for $S_v$ conditioned to the property that the cycle $d_1,\ldots,d_t$ does not contain any crossing.

		
		Set the side containing $B_0$ to be the outer face, and the other to be the inside face of $v$ (unless $v=v_0$, in which case $B_0$ is in the inner face).




%
%
		
	
	\end{enumerate}
	\item \label{step.bip1}Given a $2$-connected block $B$ with root $v$,
	order the vertices of $B$ according to their distance to $v$. Let $l(v)$ be the label of $v$.
	\begin{itemize}
		\item If $B$ is to the inside of $v$, then label the vertices at distance $i$ from $v$ by $l(v)+i+1 \mod 2$ (so the neighbors of $v$ get the same label as $v$).
		\item If $B$ is to the outside of $v$, label the vertices at distance $i$ from $v$ by $l(v)+i \mod 2$.
		\item The parent of the vertex $u$ in $B$ is the minimal neighbor (in the total order given to the vertices) of $u$ strictly closer to $v_0$ than $u$.
	\end{itemize}

	\item \label{step.bip2} Orient the edges from the vertex labeled $0$ to the vertex labeled $1$ if the vertices are labeled differently. (If $M$ is bipartite, this orientation is independent on the chosen parent.)
	
	If the parent and the child have the same label, then the orientation of the edge is given by the child's label: from the child labeled $0$ to the parent labeled $0$ and from the parent labeled $1$ to the child labeled $1$.
	
	Orient the loops from the smaller dart to the larger dart, according to the total order on the darts.
	\item Return the map with the orientation on the edges.
\end{enumerate}

\subsection{Properties of the algorithm}\label{sec.alg_cor}

\begin{proposition}\label{p.alg_correct}
	Let $M$ be a connected map. Then
	\begin{enumerate}
		\item\label{prop.part.1} $\text{Alg}(M)$ gives an orientation to the edges of $M$.
		\item\label{prop.part.4} For every cut vertex $v$, Step~\ref{step.planar} chooses an embedding of $S_v$.
		Let $\{M_1,\ldots,M_q\}$ be all the possible outputs for $\text{Alg}(M)$.
		Assume that $\text{Mv}(M_1)$ is a plane map. Then 
		\begin{enumerate}
			\item \label{prop.part.4a} $\text{Mv}(M_i)$ is plane for all $i\in[q]$, 
			\item \label{prop.part.4b} the algorithm runs in linear time in $|V(M)|$. (We are assuming that, from a character in the paragraph, we may access both copies of the same character, and their combined neighbors of size $4$, in constant time.)
		\end{enumerate}
		\newcounter{enumtemp}
		\setcounter{enumtemp}{\theenumi}
	\end{enumerate}
%
Let $\omega$ be a Gauss paragraph. Let $M=\text{Sf}(\omega)$. Then
\begin{enumerate}	
	\setcounter{enumi}{\theenumtemp}
		\item \label{prop.part.2} 
		$M$ is bipartite.
		\item\label{prop.part.3} 
		For each cut vertex $v$ of $M$, $S_v$ (as defined in Step~\ref{step.def_sv}) is a planar graph.

	\end{enumerate} 
\end{proposition}

\begin{proof}[Proof of Proposition~\ref{p.alg_correct}]
	
	\emph{Part \ref{prop.part.1}.} The algorithm performs 
	a test of bipartiteness on $M$ and planarity for each $S_v$. However, the algorithm does not stop when the test fail and it always gives an orientation of the edges. The orientation is given after the algorithm internally determines a root, a labeling on the vertices and a parent for each vertex besides the root.



%

	\emph{Part~\ref{prop.part.3}.} 
	Let us show the following claim.
	
	\begin{claim}\label{cl.2}
		Let $\sigma$ be an orientation of the edges of $M$,  $v$ a vertex of $M$, and $S_v$ be the graph defined in Step~\ref{step.def_sv}.
		If $\text{Mv}(M,\sigma)$ is a map on an orientable surface of genus $k$, then the graph $S_v$ has orientable genus $\leq k$.
		
		Moreover, if $\text{Mv}(M,\sigma)$ is plane, then $\text{Mv}(M,\sigma)$ defines a plane embedding for each $S_v$.
	\end{claim}

	\begin{proof}[Proof of Claim~\ref{cl.2}]
Let $G_0$ be the $4$-regular graph underlying the oriented map $\text{Mv}(M,\sigma)$.
Let $(d_1,\ldots,d_t)$ be the cyclic permutation of darts associated to $v$ in $M$.
 In $G_0$, this cyclic permutation associated to $v$ induces a closed circuit formed by the ordered set of vertices $(\{d_1,\alpha(d_1)\},\{d_2,\alpha(d_2)\},\ldots,\{d_t,\alpha(d_t)\})$ (see Observation~\ref{obs.prop_vm}). If $v$ is loopless, the closed circuit is a cycle on $t$ vertices.

Let $G_1$ be the graph obtained from $G_0$ by $1$-subdividing each edge in the closed circuit induced by $v$ in $G_0$: for each $i\in [1,t]$, indices modulo $t$, the edge from $\{d_i,\alpha(d_i)\}$ to $\{d_{i+1},\alpha(d_{i+1})\}$ is replaced by a path on the $3$ vertices $(\{d_i,\alpha(d_i)\},u_i,\{d_{i+1},\alpha(d_{i+1})\})$, where $u_i$ is a new vertex. From the embedding $\text{Mv}(M,\sigma)$ we can obtain 
$\text{Mv}(M,\sigma)'$, an embedding of $G_1$, by considering the vertex $u_i$ as a point in the edge $\{d_i,\alpha(d_i)\}\{d_{i+1},\alpha(d_{i+1})\}$.


		
Let $G_2$ be the graph derived from $G_1$ by adding an edge $e_i$ between $u_i$ and $u_{i+1}$, for each $i\in[1,t]$, indices modulo $t$.
In $\text{Mv}(M,\sigma)$, the edge $\{d_i,\alpha(d_i)\} \{d_{i+1},\alpha(d_{i+1})\}$, and the edge  $\{d_{i+1},\alpha(d_{i+1})\}\{d_{i+2},\alpha(d_{i+2})\}$ are contiguous in the cyclic orientation around the vertex $\{d_{i+1},\alpha(d_{i+1})\}$ (see Observation~\ref{obs.prop_vm}). Hence, these edges are consecutive in a face of $\text{Mv}(M,\sigma)$. This implies that each edge $e_i$ can be added to the map $\text{Mv}(M,\sigma)'$ to create a map $\text{Mv}(M,\sigma)''$ which embeds $G_2$ in the same surface as $\text{Mv}(M,\sigma)$. In particular, the genus of $G_2$ is less or equal than the genus of $\text{Mv}(M,\sigma)$.

%
%
%
		
		Let $G_3$ be obtained by deleting, for each $i\in[1,t]$, the edge from $u_i$ to $\{d_{i+1},\alpha(d_{i+1})\}$ in $G_2$.
		Let $G_4$ be obtained from $G_3$ by contracting all the edges not attached to the vertices $\{u_i\}_{i\in[1,t]}$.
		Then $S_v=G_4$, and $S_v$ can be embedded in the surface of $\text{Mv}(M,\sigma)$ or one with lower genus (as $G_4$ is a minor of $G_2$). Thus both parts of the claim follow.
	\end{proof}

By Proposition~\ref{prop.inv_op}, the plane embedding $\Omega$ of $\omega$ induces an orientation $\sigma$ to the edges of $\text{Sf}(\omega)$ for which $\text{Mv}(\text{Sf}(\omega),\sigma)$ is plane. Hence, Claim~\ref{cl.2} shows that each $S_v$ is planar.

	\emph{Part~\ref{prop.part.4a}.} Let us show the following lemma.
	
	\begin{lemma}\label{lemma.dif_planar}
		Let $\{M_1,\ldots,M_q\}$ be all the possible oriented maps output by $\text{Alg}(M)$, due to possible choices of Step~\ref{step.planar}.
		Assume that $\text{Mv}(M_1)$ is a plane map. Then $\text{Mv}(M_i)$ is plane for all $i\in[1,q]$.
	\end{lemma}
\begin{proof}[Proof of Lemma~\ref{lemma.dif_planar}]
		In $\text{Alg}(M)$, the different choices in Step~\ref{step.planar} only affect the different orientations of the edges of $M$.
		The fact that $\text{Mv}(M_1)$ is plane implies that 
		 the double occurrence paragraph $\omega$ induced by $\text{Mv}(M_1)$ is realizable (is a Gauss paragraph). Thus, $(M,\sigma)=\text{Sf}(\text{Mv}(M_1))$ is bipartite (see Part~\ref{prop.part.2}) and $M$ has no loops.
		Additionally, Part~\ref{prop.part.3} shows that each $S_v$ is planar.
		
%

If $v_1$ and $v_2$ are two cut vertices, then the choices for plane embeddings for $S_{v_1}$ and for $S_{v_2}$ are independent. 
Hence the claim follows if we show that, given a generic cut vertex $v$, then $M_2$, obtained from $M_1$ by a change in the plane embedding for $S_v$, is such that $\text{Mv}(M_2)$ is plane. Let $S_1$, resp. $S_2$, be the plane embedding of $S_v$ corresponding to $M_1$, resp. $M_2$.

Let $(d_1,\ldots,d_t)$ be the cyclic permutation associated to the vertex $v$ in $M$, and let $\{B_0,B_1,\ldots,B_k\}$ be the vertices in $S_v$ different from $d_1,\ldots,d_t$ (we follow the notation from the algorithm $\text{Alg}(\cdot)$). Since $M$ has no loops, then $S_v$ has no vertices $L_i$. For simplicity, the vertex $\{d_i,\alpha(d_i)\}$ is denoted by $d_i$.
The inside/outside status of the vertices $\{B_0,B_1,\ldots,B_k\}$ (with respect to the cycle induced by $d_1,\ldots,d_t$) completely determines the embedding of $S_v$.


Given $\mathcal{B}$ a subset of $\{B_0,B_1,\ldots,B_k\}$, let $N(\mathcal{B})$ denote the set of neighbors of $\mathcal{B}$ among $d_1,\ldots,d_t$.
Let $\{B_1,\ldots,B_l\}$ be the set of vertices that change its status from $S_1$ to $S_2$.

\begin{observation}\label{obs.aux.1}
	Let $B\in \{B_1,\ldots,B_l\}$ with $d_i,d_j\in N(B)$ and $i<j$. Let $B'\in \{B_0,B_{l+1},\ldots,B_k\}$ with $d_{i'},d_{j'}\in N(B')$.
	\begin{itemize}
		\item 
	Then either $i',j'\in (i,j)$ or $i',j'\in [1,i)\cup (j,t]$. Indeed, say $i'\in (i,j)$ and $j'\in [1,i)\cup (j,t]$, then the paths $(d_i,B,d_j)$ and $(d_i',B',d_j')$ should be on different sides with respect to the cycle $(d_1,\ldots,d_t)$. The path $(d_i,B,d_j)$ changes sides (in $S_1$ with respect to $S_2$) while $(d_i',B',d_j')$ remains in on the same side on both $S_1$ and $S_2$. Thus, we obtain a contradiction with the planarity of $S_1$ and $S_2$, hence the claim that either $i',j'\in (i,j)$ or $i',j'\in [1,i)\cup (j,t]$ follows. The case $j'\in (i,j)$ and $i'\in [1,i)\cup (j,t]$ follows similarly.
	
	
	\item Moreover, let $B''\in \{B_1,\ldots,B_l\}$ with
	$d_{i''},d_{j''}\in N(B'')$. Assume that $[i,j]\cap [i'',j'']\neq \emptyset$ (the intervals are modulo $t$, so if $i''<j''$ we either pick 
	$[i'',j'']$ or $[j'',t+i'']=_{\mod t}[j'',i'']$ so that the condition $[i,j]\cap [i'',j'']\neq \emptyset$ is satisfied).
	
	If $B'\in \{B_0,B_{l+1},\ldots,B_{k}\}$ is such that $N(B')\cap [i,j]\cap [i'',j'']\neq \emptyset$, then, by the previous argument, $N(B')\subset [i,j]\cap [i'',j'']$.
	
	\end{itemize}
\end{observation}

\begin{observation}\label{obs.aux.2}
If $\{B_{i_1},\ldots,B_{i_l},L_{j_1},\ldots,L_{j_{l'}}\}$ is a subset of the vertices in $S_v$ whose neighborhood, $N(\{B_{i_1},\ldots,B_{i_l},L_{j_1},\ldots,L_{j_{l'}}\})$, is a set of consecutive vertices from $d_1,\ldots,d_t$, then, in any plane embedding of $S_v$, we can switch (all at once) the inside/outside status of 
$\{B_{i_1},\ldots,B_{i_l},L_{j_1},\ldots,L_{j_{l'}}\}$ and obtain another plane embedding of $S_v$, independently on the inside/outside status of the other vertices.
\end{observation}


\begin{claim}\label{cl.4}
	Let $v$ a cut vertex of $M$, with darts $d_1,\ldots,d_t$, that induces the connected components $B_0,\ldots,B_k$ and with loops $L_1,\ldots,L_r$ (see Step~\ref{step.def_sv}.) Let $S_1$ and $S_2$ be two plane embeddings of $S_v$, where $B_0$ is in the exterior face. Then there is a sequence of plane embeddings of $S_v$, $s_1,\ldots,s_p$, with $S_1=s_1$, $s_p=S_2$ with the following property.
	Each $s_i$ is obtained from $s_{i+1}$ by switching the inside/outside status of a subset $\mathcal{B}_{i,i+1}$ of vertices from $B_0,B_1,\ldots,B_k,L_1,\ldots,L_r$ such that the neighborhood of $\mathcal{B}_{i,i+1}$ in $S_v$ is a consecutive (in the cyclic order) subset of $d_1,\ldots,d_t$.
%
%
%
%
%
%
%
%
\end{claim}

\begin{proof}[Proof of Claim~\ref{cl.4}] Since in $S_v$ the vertex corresponding to the loop $L_i=\{d,\alpha(d)\}$ behaves equivalently as if there was a $2$-connected component $B_j$ connected to $d$ and $\alpha(d)$, we may assume, without loss of generality, that $r=0$.
	
%

Let $\{B_1,\ldots,B_l\}$ be the switched vertices from $S_1$ to $S_2$.
 Let $\{d_1,\ldots,d_p\}$ denote the minimal set of consecutive darts containing $N(\{B_1,\ldots,B_l\})$.
Then $\{d_1,\ldots,d_p\}$ can be broken into consecutive intervals $I_1,\ldots,I_{2q+1}$ such that, for $i\in [0,q]$,
\[
\begin{cases}
I_{2i+1}\cap N(\{B_1,\ldots,B_l\})=I_{2i+1} \\
I_{2i}\cap N(\{B_1,\ldots,B_l\})=\emptyset
\end{cases}
\]
The claim is shown by an induction on $q$.
The base case $q=0$, follows directly from Observation~\ref{obs.aux.2} with $\mathcal{B}_{1,2}=\{B_1,\ldots,B_l\}$.

Assume $q\geq 1$. 
Let $\mathcal{P}\subset {[0,q] \choose 2}$ be the set of pairs of indices $(i,j)$ for which there exists a vertex in $B_1,\ldots,B_l$ having neighbors both in $I_{2i+1}$ and in $I_{2j+1}$.
If there exists an index $i$ with no pair of the type $(i,j)\in \mathcal{P}$, then there exists a subset $\mathcal{B}\subset \{B_1,\ldots,B_l\}$ only connected to $I_{2i+1}$, so $N(\mathcal{B})=I_{2i+1}$. We can remove $\mathcal{B}$ from $\{B_1,\ldots,B_l\}$, assume the result by induction on $q-1$, and then add another inside/outside switch concerning exactly $\mathcal{B}$. In particular, we may assume that $\mathcal{P}$ is non-empty.

Let $(i_0,j_0)\in \mathcal{P}$, $i_0<j_0$, be a pair with minimal separation (minimal $|i_0-j_0|$). The induction follows after showing that there exists a set $\mathcal{B}\subset \{B_0,B_{l+1},\ldots,B_{k}\}$ such that $N(\mathcal{B})=I_{2(i_0)+2}$. Indeed, by Observation~\ref{obs.aux.2} shifting $\{B_1,\ldots,B_l\}\cup \mathcal{B}$ in $S_1$ creates another plane embedding $S_{3}$ of $S_v$ with the property that $N(\{B_1,\ldots,B_l\}\cup \mathcal{B})$ has at least two less intervals. The sequence of changes exists then by induction from $S_1$ to $S_3$. By switching again the whole set $\mathcal{B}$ in $S_3$ we obtain $S_2$ and the result follows.


Assume $j_0-i_0=1$. The first part of Observation~\ref{obs.aux.1}, and the fact that $N(\{B_0,B_{l+1},\ldots,B_{k}\})\cap (I_{2(i_0)+1}\cup I_{2(j_0)+1})=\emptyset$, shows the existence of the set $\mathcal{B}\subset \{B_0,B_{l+1},\ldots,B_{k}\}$ with $N(\mathcal{B})=I_{2(i_0)+2}$.

Assume now $|i_0-j_0|>1$. The minimality of $(i_0,j_0)$ shows the existence of a $B''\in \{B_1,\ldots,B_l\}$ such that $\emptyset \neq  N(B'')\cap I_{2(i_0)+3}\owns d_i$
and $ \emptyset\neq N(B'')\cap \linebreak[1](\{d_1,\ldots,d_t\}\setminus \linebreak[0]\{I_{2(i_0)+1},\ldots,\linebreak[0]I_{2(j_0)+1}\})\owns\linebreak[0] d_j$. Then either
\[
\{d_i,d_{i+1},\ldots,d_j\}\cap \{I_{2(i_0)+1},\ldots,I_{2(j_0)+1}\}\subset \{I_{2(i_0)+1},I_{2(i_0)+2},I_{2(i_0)+3}\}\]
or
\[\{d_j,d_{j+1},\ldots,d_i\}\cap \{I_{2(i_0)+1},\ldots,I_{2(j_0)+1}\}\subset \{I_{2(i_0)+1},I_{2(i_0)+2},I_{2(i_0)+3}\},\]
with indices in $d_{\cdot}$ modulo $t$. In either of the cases, we use Observation~\ref{obs.aux.1}, and the fact that $N(\{B_0,B_{l+1},\ldots,B_{k}\})\cap (I_{2(i_0)+1}\cup I_{2(j_0)+1})=\emptyset$, to find the sought after set $\mathcal{B}\subset \{B_0,B_{l+1},\ldots,B_{k}\}$ with $N(\mathcal{B})=I_{2(i_0)+2}$. This finishes the induction step and the proof of Claim~\ref{cl.4}.
\end{proof}

By Claim~\ref{cl.4}, we may assume that $S_1$ differs from $S_2$
by an inside/outside status switch of the vertices $\{B_1,\ldots,B_l\}$ which are attached to the consecutive set of darts $\{d_1,d_2,\ldots,d_j\}$.

Let $C_1$, resp. $C_2$, be the map induced by the vertices in $\text{Mv}(M_1)$, resp. $\text{Mv}(M_2)$, coming from edges in the connected components $\{B_1,\ldots,B_l\}$. Let $C_0$ be the map induced in $\text{Mv}(M_1)$ and $\text{Mv}(M_2)$ by the edges in the components $B_0,B_{l+1},\ldots,B_k$.
Since $B_0,B_1,\ldots,B_k$ induce a partition on the edges in the map $M$, $C_0$ along with $C_1$, resp. $C_2$, induce a partition on the vertices of $\text{Mv}(M_1)$, resp. $\text{Mv}(M_2)$.

Recall that the edges in $\text{Mv}(M_i)$ correspond to pairs of darts consecutive in the vertex permutation $\tau$ of $M_i$ (notice that the next dart of $d$ may be $d$ itself). Hence, 
\begin{itemize}
	\item 
if two darts are consecutive in a vertex rotation other than the corresponding to $v$, its corresponding edge belongs entirely to $C_0$ or to $C_1$ (or to $C_2$). 
\item For $i\in [1,j-1]$, both consecutive edges in $M$, $\{d_i,\alpha(d_i)\}$ and $\{d_{i+1},\alpha(d_{i+1})\}$ belong to $\{B_1,\ldots,B_l\}$, hence the corresponding edge lie between two vertices of $C_1$ in $\text{Mv}(M_1)$, resp. $C_2$ in $\text{Mv}(M_2)$.
\item For $i\in [j,t-1]$, both consecutive edges in $M$, $\{d_i,\alpha(d_i)\}$ and $\{d_{i+1},\alpha(d_{i+1})\}$ belong to $\{B_0,B_{l+1},\ldots,B_{k}\}$, hence the corresponding edge lie between two vertices of $C_1$ in $\text{Mv}(M_1)$, resp. $C_2$ in $\text{Mv}(M_2)$.
\item The edges corresponding to the consecutive pairs of darts $(d_{j},d_{j+1})$ and $(d_t,d_1)$ are exactly the edges between $C_0$ and $C_1$, resp. $C_2$.
In particular, they connect the vertex $\{d_{j},\alpha(d_{j})\}\in C_1$, resp. $C_2$, to the vertex $\{d_{j+1},\alpha(d_{j+1})\}\in C_0$, and
the vertex $\{d_{t},\alpha(d_{t})\}\in C_0$ to the vertex $\{d_{1},\alpha(d_{1})\}\in C_1$, resp. $C_2$.
\end{itemize}

The darts $d_t,d_1,\ldots,d_j,d_{j+1}$ from $v$ induce a not self-intersecting path in $\text{Mv}(M_1)$. Remove all the vertices in $C_1$
with the exception of those coming from the edges 
$\{\{d_1,\alpha(d_1)\},\linebreak[0]\ldots,\linebreak[0]\{d_j,\alpha(d_j)\}\}$ in $M$. After the removal, the vertices
$\{\{d_1,\alpha(d_1)\},\ldots,\{d_j,\alpha(d_j)\}\}$ have degree exactly two. In particular, they form a chord in a face $F$. Therefore, $C_1$ and $C_0$ are plane induced submaps and $C_1$ sits inside the bounded face related to $F$ in $C_0$. Then, the exterior face of $C_1$ contains the vertices $\{d_{j},\alpha(d_{j})\}$ and 
$\{d_{1},\alpha(d_{1})\}$, as these are the vertices in $C_1$ that have edges connected to $C_0$, and the face in $C_0$ where $C_1$ lies contains the vertices $\{d_{j+1},\alpha(d_{j+1})\}$ and $\{d_{t},\alpha(d_{t})\}$.

%

Observe that the orientation given by the algorithm $\text{Alg}(\cdot)$ of the edges from $\{B_1,\ldots,B_l\}$ in $M_1$ is exactly the reverse orientation than in $M_2$. Therefore, the cyclic permutations of the edges around the vertices in $C_1$ are exactly the inverses of those in $C_2$. In particular, $C_2$ is also a plane submap (consider the end of the edges from $C_1$, or $C_2$, to $C_0$ as leaf vertices), with the same exterior face as $C_1$ (the part of $C_0$ have remained invariant from $\text{Mv}(M_1)$ to $\text{Mv}(M_1)$).
In particular, any face of $C_1$ remains a face in $C_2$: the internal faces remain internal, and the external face remains external. Thus, the partition of the external face into two parts (between the vertex $d_1$ and $d_j$ and between $d_j$ and $d_1$) remains, but the directions are reversed.
That is to say, if in $\text{Mv}(M_1)$ the two faces that involved vertices from $C_0$ and $C_1$ were $f_1=d_1,[\text{vertices in }C_0]_0,d_j,[\text{vertices in }C_1]_1, d_1$ and $f_2=d_j,[\text{vertices in }C_0]_2,d_1,[\text{vertices in }C_1]_3, d_j$, then in $\text{Mv}(M_2)$ the corresponding faces are
$f_1'=d_1,[\text{vertices in }C_0]_0,d_j,[\text{vertices in }C_1]_3^{-1}, d_1$ and $f_2=d_j,[\text{vertices in }C_0]_2,v_1,[\text{vertices in }C_1]_1^{-1}, d_j$. The other faces in $C_0$ remain untouched when seen in $\text{Mv}(M_2)$. Since the number of vertices, faces, and edges remains unaltered from $\text{Mv}(M_1)$ to $\text{Mv}(M_2)$, then $\text{Mv}(M_2)$ is plane as well (use Euler's formula to determine the genus). This finishes the proof of Lemma~\ref{lemma.dif_planar}.
\end{proof}

\emph{Part~\ref{prop.part.4b}.} 
	Let $G$ be the $4$-regular graph underlying the paragraph $\omega$. Since the number of edges of $G$ is linear with respect to the number of characters of $\omega$,
	a graph structure for $G$ in which visiting a neighbor of a vertex in $G$ takes constant time can be obtained in linear time in the number of characters of the $\omega$. The connected components of $G$ can be obtained then in linear time in the number of vertices using a deep-first search. Then the Seifert map of $\omega$ can be obtained in linear time using the graph $G$. Giving a total order on the edges and darts of $M=\text{Sf}(\omega)$ costs linear time.
	Checking bipartitness can also be done in linear time as the number of edges is linear.
	Computing the block tree decomposition of $M$, along with its cut vertices, can be computed in linear time in the number of vertices \cite{hopcroft1973algorithm}.
	
	By Part \ref{prop.part.3}, for every cut vertex $v$, $S_v$ is planar. Finding a planar embedding of a graph can be computed in linear time in the size of its vertex set \cite{hopcroft1974efficient}. The number of vertices of $S_v$ is proportional to the number of darts attached to $v$, and each dart is treated at most once (if the dart is attached to a cut vertex $v$). Therefore, the total number of vertices treated for all the planarity tests is proportional to the number of darts, which is linear in the number of vertices (as otherwise $S_v$ is not planar for some $v$).
	Furthermore, if each $S_v$ is planar, the algorithm does not compute
	a minimal crossing number, which might be costly.
	
	Assigning a parent and orienting the edges of $M$ can be done in linear time in the number of edges, which is the same as the number of characters of $\omega$ if $M=\text{Sf}(\omega)$.
	Checking that a map is planar can be done in linear time in the number of edges (hence linear time in the number of vertices).
	 Thus, the whole algorithm performs in linear time when $\omega$ is a Gauss paragraph.

	 \emph{Part \ref{prop.part.2}} follows form Proposition~\ref{prop.clar}. The clockwise/counterclockwise orientation of the Seifert cycles in the plane depends on the orientation on the neighboring cycles and on its relative containment (see the chart in Proposition~\ref{prop.clar}). As mentioned at the beginning of Section~\ref{sec.char}, the label $0$ or $1$ on the vertices of $M$ (which are the Seifert cycles of $\omega$ if $M=\text{Sf}(\omega)$) indicate the clockwise/counterclockwise of the cycle. As all the Seifert cycles in a $2$-connected component of $M$ are pairwise outside from each other (with the exception of the root of the $2$-connected component if the component lies inside of the root), then no odd cycles appear in $\text{Sf}(\omega)$.
\end{proof}

\begin{observation}\label{obs.counter}
	The double occurrence word $\omega=1234534125$ is such that $S_v$ is planar for every $v$, $\text{Sf}(\omega)$ is bipartite, and it satisfies that between the two copies of the same characters, there is an even number of characters. However, one can check that $\omega=1234534125$ is not a Gauss code.
\end{observation}

\subsection{Characterizing Gauss codes}\label{sec.main_proof}

Before proceeding to the main result, let us show a helpful proposition.

\begin{proposition}\label{prop.clar}
	Let $\omega$ be a Gauss paragraph, and let $\Omega$ be a plane embedding of $\omega$. The plane is assumed to be oriented. Then
	\begin{itemize}
		\item Every Seifert cycle is a close curve with no self intersections. In particular, every Seifert cycle in $\Omega$ is either oriented clockwise or counterclockwise (traversed in the direction given by the orientation on the edges).
		
		\item  No pair of Seifert cycles traverse each other.

		\item Let $c_1$, $c_2$ be two Seifert cycles intersecting at $v$. Let $c_{1,+}$ ($c_{1,-}$) and $c_{2,+}$ ($c_{2,-}$) be the darts going to/positive (leaving/negative) $v$ from $c_{1}$ and $c_2$ receptively Then we have the following properties,
		\begin{center}
			\begin{tabular}{|p{2cm}|c|c|c|p{3cm}|}
				\hline
				in./out. & or. of	$c_1$ & $\Rightarrow$ or. of $c_2$  & $\Rightarrow$ permutation of $v$ in $\Omega$ & $\Rightarrow$ orientation of edge $v$ in $\text{Sf}(\Omega)$ \\
				\hline
				$c_1$ inside $c_2$	& counter. & counter. & $ (c_{1,+},c_{2,+},c_{2,-},c_{1,-})$ & from $c_2$ to $c_1$ \\
				\hline
				$c_1$ inside $c_2$	& clock. & clock. & $(c_{2,+},c_{1,+},c_{1,-},c_{2,-})$ & from $c_1$ to $c_2$ \\
				\hline
				$c_1$ and $c_2$ outside & counter. & clock.&
				$(c_{1,+},c_{2,+},c_{2,-},c_{1,-})$ & from $c_2$ to $c_1$ \\
				\hline
				$c_1$ and $c_2$ outside & clock. & counter.&
				$(c_{2,+},c_{1,+},c_{1,-},c_{2,-})$ & from $c_1$ to $c_2$ \\
				\hline
			\end{tabular}
		\end{center}
		with the addition of the symmetries when $c_2$ is inside $c_1$.
		
	\end{itemize}
\end{proposition}

\begin{proof}[Proof of Proposition~\ref{prop.clar}]
	
	The first and second parts follows by observing that each Seifert cycle is obtained by taking, after arriving to an intersection, the neighboring outgoing edge, instead of the ``straight ahead'' one. In particular, a Seifert cycle is a cycle in any valid plane embedding of $\omega$. Furthermore, the Seifert cycles divide the four edges concurrent into a vertex in pairs such that the two paths concurrent at a vertex do not intersect transversaly (they share a vertex but each path stays locally on the same side with respect to the other). Therefore, no pair of Seifert cycles intersect. This also applies to the intersection of a Seifert cycle with itself.

	In the plane each non-self-intersecting closed curve defines an exterior  and an interior faces. If two of these curves that do not traverse the other share a vertex, then the vertex is point of tangency between the curves. Since in the local cyclic orientation of the edges the pair of incoming and outgoing edges are consecutive, both directions of the cycles are concurrent at the tangency (shared) point. Then the inside/outside relation of the closed curves, together with the orientation of one of the cycle, determines the orientation of the other cycle (as well as the local embedding of the edges around the vertex given of a $2/2$-premap). The relations are given in the chart.
\end{proof}


\begin{theorem}\label{t.characterization}
	Let $\omega$ be a doubly occurrence paragraph. If $\omega$ is embeddable on the plane (a Gauss paragraph) then 
	\begin{itemize}
		\item The procedure $\text{Alg}(\text{Sf}(\omega))$ gives an orientation on the edges, $\sigma$, such that
		$\text{Mv}(\text{Sf}(\omega),\sigma)$ is a plane map and
		\item  Recording the traversed vertices using the ``straight ahead'' edge as the next edge for $\text{Mv}(\text{Sf}(\omega),\sigma)$ gives the paragraph $\omega$.
	\end{itemize}
	
	If $\omega$ is not embeddable on the plane (not a Gauss paragraph), then the map $\text{Mv}(\text{Sf}(\omega),\sigma)$ is not plane.
\end{theorem}

\begin{proof}

Part~\ref{prop.part.3} in Proposition~\ref{p.alg_correct} shows that if $\omega$ is a Gauss code or paragraph, then each $S_v$ as defined in Step~\ref{step.def_sv} from $\text{Sf}(\omega)$ is planar. 

By iteratively picking a Seifert cycle contained in another Seifert cycle if the original one was not a face, any plane embedding of $\omega$ has a Seifert cycle that is a face.
Consider a plane embedding of $\omega$ so that a Seifert cycle is the exterior face and it is oriented counterclockwise. 
 Let $\Omega$ be  such embedding.

 By Observation~\ref{obs.orient_premap_from_word}, Observation~\ref{obs.algo_algo}, Part~\ref{p.bij-orient-map} of Proposition~\ref{prop.inv_op}, and the definition of the Seifert map of a paragraph (see the paragraph after Definition~\ref{d.seif_map}),  $\Omega$ belongs to the $2/2$-premap $\text{Mv}(\text{Sf}(\omega))$.
 
By the second part of Claim~\ref{cl.2},  $\Omega$ induces a plane embeddings for all $S_v$.
Using Proposition~\ref{p.alg_correct} Part~\ref{prop.part.4a}, the planarity of
$\text{Mv}(\text{Sf}(\omega),\sigma)$ does not depend on the orientation $\sigma$ output by the algorithm.
Hence, it suffices to show that the orientation $\sigma$ on $\text{Sf}(\omega)$, given by the algorithm using the choice (in Step~\ref{step.planar}) of the plane embeddings of $S_v$ induced by $\Omega$, is such that $\text{Mv}(\text{Sf}(\omega),\sigma)$ is plane.
This follows from Proposition~\ref{prop.clar} by observing that the relation being inside/outside together with the orientation of a cycle is transmitted to the neighboring cycles. The relation is tracked by the embeddings of the graphs $S_v$ (inside/outside vertices to $v$) and the distance of a Seifert cycle to the root in its $2$-connected component
and is reproduced in Step~\ref{step.bip1} of $\text{Alg}(M)$.
The cyclic orientation of the edges in the vertices $u$ of $\Omega$ may be deduced from the orientation and relative inclusion between the two Seifert cycles meeting at the vertex $u$ (see the chart in Proposition~\ref{prop.clar}), and are precisely the ones given by the orientation on the edges of $\text{Sf}(\omega)$ when the embeddings on the $S_v$'s induced by $\Omega$ are considered (see Step~\ref{step.bip2} of $\text{Alg}(M)$).


%
%
%
%
%

%
%



Since the operations are invertible, any map in  $\text{Mv}(\text{Sf}(\omega))$ gives a valid embedding of $\omega$ (where the decomposition in closed circuits following the ``straight ahead'' edge walks gives the words of $\omega$). If $\omega$ is not a Gauss paragraph, $\text{Mv}(\text{Sf}(\omega),\sigma)$ cannot be plane.

\end{proof}


\subsection{On all the plane embeddings}

\begin{proposition}\label{prop.give_all_embeddings}
	Let $\omega$ be a Gauss paragraph. Let $\{v_1,\ldots,v_k\}$ be the cut vertices of $M=\text{Sf}(\omega)$ and let $v_0$ be the root of $M$ (a non-cut vertex). Let $S_{v_i}$ be the graphs associated to the cut vertices $v_1,\ldots,v_k$ as defined in Step~\ref{step.def_sv}. Then
	\begin{itemize}
		\item Any collection of plane embeddings of $S_{v_1},\ldots,S_{v_k}$ (where the vertex associated to the root of the graph $v_0$ is in the exterior face of the cycle $v_i$), gives a plane embedding of $\omega$.
		\item Each plane embedding of $\omega$ is given by a collection of plane embeddings of $S_{v_1},\ldots,S_{v_k}$ where the vertex associated to the root of the graph is in the exterior face of the cycle $v_i$.
	\end{itemize}
\end{proposition}

The proof of the first point follows from the Part~\ref{prop.part.4a} in Proposition~\ref{p.alg_correct} after applying the first part of Theorem~\ref{t.characterization}. The second part follows as any plane embedding of $\omega$ induces a plane embedding for each $S_{v_i}$, and these plane embeddings induce the original embedding using the orientation on the edges giving by the algorithm (see the proof of Theorem~\ref{t.characterization}).

\bibliographystyle{plain}
\bibliography{biblio.bib}

\end{document}